\documentclass[12pt]{article}

\RequirePackage{fancymath}

\usepackage{xcolor}
\usepackage{soul}
\usepackage{tikz}
\usepackage{pgfplots}
\usepackage{latexsym,amsmath,enumerate,graphics,enumerate,amsthm,tikz,hyperref,float}
\usepackage[mathscr]{euscript}
\usepackage[affil-it]{authblk}
\usepackage{enumerate,amsthm,dsfont,pstricks}
\usepackage{latexsym,amsmath,amssymb,amscd,wrapfig,graphicx}
\usepackage{hyperref,cleveref}
\usepackage{enumitem}
\usepackage{circuitikz}
\usepackage{bbm}

\makeatletter
\def\@maketitle{%
  \newpage
  \null
  \begin{center}%
  \let \footnote \thanks
    {\Large\bfseries \@title \par}%
    \vskip 1.5em%
    {\normalsize
      \lineskip .5em%
      \begin{tabular}[t]{c}%
        \@author
      \end{tabular}\par}%
    \vskip 0.8em%
    {\small \@date}%
  \end{center}%
  \par
  \vskip 1.5em}
\makeatother

\title{Classical theorems from analysis for locally band preserving functions on Dedekind complete $\Phi$-algebras}

\author{Eder Kikianty%
\thanks{Email: \texttt{eder.kikianty@wits.ac.za}}}
\affil[1]{School of Mathematics, University of Witwatersrand, Private Bag 3, WITS 2050, South Africa}

\author[2]{Luan Naude%
\thanks{Email: \texttt{u23969998@tuks.co.za}}}
\affil[2]{Department of Mathematics and Applied Mathematics, University of Pretoria, Private Bag X20, Hatfield 0028, South Africa}

\author{Mark Roelands%
\thanks{Email: \texttt{m.roelands@math.leidenuniv.nl}}}
\affil[3]{Mathematical Institute, Leiden University, 2300 RA Leiden,
The Netherlands}

\author[2]{Christopher Schwanke%
\thanks{Email: \texttt{cmschwanke26@gmail.com}}}
\date{\today}

\begin{document}

\maketitle

\vspace{-.8cm}

\begin{abstract}
{\scriptsize In this paper we explore the concept of locally band preserving functions, introduced by Ercan and Wickstead, on Dedekind complete $\Phi$-algebras. Specifically, we show that all super order differentiable functions are locally band preserving. Furthermore, some foundational results from classical analysis are proved in this setting, such as the Intermediate Value Theorem, the Extreme Value Theorem, and the Mean Value Theorem. Moreover, we show that these generalisations can fail for functions that are not locally band preserving. With the goal in mind to further develop the theory of complex differentiation in Dedekind complete complex $\Phi$-algebras, a complex version of the Mean Value Theorem is also provided.   }  
\end{abstract}



\section{Introduction}

The theory of complex analysis on Banach spaces enjoys a long, well-known,
and fruitful history, which is summarised in the text \cite{Mujica}.
This theory generalises the
classical theory of functions of a complex variable by using norms as an abstraction of the
ordinary modulus on the complex plane.

Recently, the papers \cite{series, differentiation} took an alternative approach to complex analysis in abstract spaces: an \textit{order-theoretic} approach. Indeed, \cite{series} began this journey by proving an $n$th root test for the absolute order convergence of series in universally complete complex $\Phi$-algebras (unital $f$-algebras). From this result, Cauchy-Hadamard formulas for calculating the radius of convergence for power series in this setting were obtained. 

This order-theoretic complex analysis was furthered in \cite{differentiation}, where the authors introduced two notions of differentiability built from the notion of order convergence in Dedekind complete $\Phi$-algebras: \textit{order differentiability} and the much stronger \textit{super order differentiability}.

Such results enable one to meaningfully investigate analytical properties of certain spaces which do not have a natural norm, such as $C^\infty_{\mathbb{C}}(X)$. They furthermore illustrate how fundamental concepts in vector lattice theory, such as weak order units and band projections, have purposeful roles to play in complex analysis. Thirdly, this theory, through analogue, sheds light on an order-theoretical \textit{real} analysis counterpart. It is to this real analysis complement that we devote much of our attention in this paper.

In Sections 2 and 3 we introduce $\epsilon$-$\delta$ reformulations of order continuity, order differentiability, and super order differentiability, with the goal in mind to have the results in this paper resemble the classical theory. 

Using these reformulations, the objective of Sections 4 and 5 is to prove analogues of a Boundedness Theorem, an Intermediate Value Theorem, an Extreme Value Theorem, and a Mean Value Theorem in Dedekind complete \textit{real} $\Phi$-algebras. It turns out, however,  that order continuity on order intervals is insufficient to prove these analogues. Therefore, in Section~4 we consider \textit{locally band preserving functions}, which were introduced by Ercan and Wickstead in \cite{Wickstead-Zafer}. The reader can also infer the similarities between locally band preserving functions and functions with the so-called local property, introduced by Avil\'{e}s and Zapata in \cite{Aviles-Zapata}, where $\Phi$-algebras of measurable functions are considered. These locally band preserving functions prove to be a natural class of functions to consider, since it will be shown that all super order differentiable functions are locally band preserving. Furthermore, it is also demonstrated that the Boundedness Theorem, the Intermediate Value Theorem, and the Extreme Value Theorem can fail for functions that are not locally band preserving.

In Section~5 we prove the analogues of the aforementioned classical theorems for order continuous and locally band preserving functions defined on a closed order interval. Using our Extreme Value Theorem, we then provide an analogue of Rolle's Theorem and of the Mean Value Theorem. As a consequence, a super order differentiable function that has derivative identically zero on an order open interval is constant on said interval. Additionally, we illustrate through example that (i) this last result can fail for functions that are order differentiable but not super order differentiable, and (ii) that an order differentiable function need not be locally band preserving. 

Finally, in order to further develop the theory of complex differentiation in Dedekind complete complex $\Phi$-algebras, a complex version of the Mean Value Theorem is proved. Consequently, any complex super order differentiable function with zero derivative on an order neighbourhood must be constant.     

\section{Preliminaries}

The reader is referred to the standard texts \cite{aliprantis, zaanen1, depagter, babyzaanen} for any unexplained terminology or basic results in vector lattice theory.

A $\Phi$-algebra is an $f$-algebra with a multiplicative unit, which will be denoted by $e$. Given a real vector lattice $F$, the vector space complexification $E := F + i F$ is equipped with a modulus  

\begin{equation}\label{e:modulus}
|x + iy| := \sup\{(\cos\theta)x + (\sin\theta)y \colon \theta \in [0,2\pi]\} \quad (x,y \in F),
\end{equation}
that extends the absolute value on $F$, provided that $F$ is closed under the supremum in \eqref{e:modulus}. In this case we call $E$ a complex vector lattice. See e.g. \cite[Section~2]{Schipper} for more details. Naturally, we say that a complex vector lattice $E=F+iF$ is Dedekind complete (respectively, has the projection property, respectively has sufficiently many projections) if $F$ is Dedekind complete (respectively, has the projection property, respectively has sufficiently many projections). In the case that $F$ is a $\Phi$-algebra, the multiplication on $F$ canonically extends to all of $E$. Thus, it is fitting for $E$ to be called a complex $\Phi$-algebra. Note that the multiplication on Archimedean $\Phi$-algebras is commutative by \cite[Theorem~2.56]{aliprantis}. We call $F$ the real part of $E$ and also write $E_\rho := F$. Furthermore, we define $E^+ := E_\rho^+$. For $x,y \in E_\rho$, we define $\Re (x + iy) := x$ and $\Im (x + iy) := y$, as usual.  It immediately follows from \eqref{e:modulus} that, for $z \in E$, we have $|\Re(z)|, |\Im(z)| \le |z|$. The reader can refer to \cite{series, differentiation} for more information about complex $\Phi$-algebras. The bands in $E$ are defined to be the complexifications of the bands in $F$, that is, $B + iB$ where $B$ is a band in $F$, see \cite[Chapter~6]{babyzaanen}. For $x\in E$, the band generated by $x$ in $E$ is the smallest band in $E$ that contains $x$ and is denoted by $B_x$. We note that $B_x$ equals the complexification of the band generated by $|x|$ in $F$. 

Often in this paper, we will consider a Dedekind complete $\Phi$-algebra $E$ that can either be real or complex. In the case that $E$ is taken to be real, we set $E_\rho := E$, and for $x \in E$, the modulus $|x|$ is the ordinary absolute value of $x$.
 
\subsection{Band decompositions from inequalities}

Because Dedekind completeness implies the projection property, every band $B$ in $E$ is a projection band. This fact means that $E = B \oplus B^d$, and, associated with $B$, there is a band projection $\bP\colon  E \to B$. The reader may assume that any notation used to denote a band carries over to its band projection (replacing the $B$ with $ \bP$), e.g. the band projections of $B$, $B_1$, and $B_x$ will be written as $\bP$, $\bP_1$, and $\bP_x$ respectively. 

We record some basic facts about bands, band projections, and weak order units in \Cref{p:properties_bands} below that we will repeatedly use throughout this paper. The proof of $(i)$ follows from \cite[Theorem~2.37]{aliprantis}, and statement $(ii)$ is a consequence of $(i)$. Statement $(iii)$ follows from \cite[Theorem~2.44]{aliprantis} and \cite[Theorem~18.13]{babyzaanen}, with $(iv)$ as a consequence. Statement $(v)$ follows from \cite[Theorem~1.45]{aliprantis}, and a proof of statement $(vi)$ can be found in \cite[Theorem~32.5]{babyzaanen}. Since the band generated by $x$ contains the band generated by $r$ whenever $0 \le r \le x$, statement $(vii)$ is evident. Since every band is a projection band in $E$, statement $(viii)$ follows. Statement $(ix)$ holds, because the sum of two bands is a band when $E$ has the projection property by \cite[Theorem~30.1(ii)]{zaanen1}, and the fact that the bands form a distributive lattice by \cite[Theorem~22.6]{zaanen1}.   

\begin{proposition} \label{p:properties_bands}
Let $E$ be a real (or complex) vector lattice, put $x, y \in E$, suppose $r$ and $s$ are weak order units of $E$, and let $B$ be a band with projection $\bP$. The following hold.
\begin{itemize}
\item[$(i)$] If $E$ is a $\Phi$-algebra, then $\bP(xy) = \bP(x) \bP(y) = \bP(x) y = \bP(y) x$,
\item[$(ii)$] if $E$ is a  $\Phi$-algebra, then $B = \bP(e) E$,
\item[$(iii)$] $\bP(r)$ is a weak order unit of $B$,
\item[$(iv)$] if $\bP(r) = 0$, then $B = \{0\}$,
\item[$(v)$] if $\bP'$ is another band projection, then $\bP \bP' = \bP' \bP$, 
\item[$(vi)$] $r \wedge s$ is a weak order unit of $E$,
\item[$(vii)$] if $x \geq r$ then $x$ is a weak order unit of $E$,
\item[$(viii)$] if $E$ has the projection property, $x$ is a weak order unit of $B$, and $y$ a weak order unit of $B^d$, then $x + y$ is a weak order unit of $E$,
\item[$(ix)$] if $E$ has the projection property, then the collection of bands forms a distributive lattice where the supremum and infimum of two bands is given by their sum and intersection, respectively.
\end{itemize}
\end{proposition}

\Cref{p:zaanen_theorem_30_5_ii} is \cite[Theorem 30.5 (ii)]{zaanen1} for real vector lattices. We add that this result easily extends to the complex case as well.

\begin{notation}
Let $E$ be a real (or complex) vector lattice. The notation $B_\alpha \uparrow B$ means that $(B_\alpha)_\alpha$ is an increasing net with supremum $B$ in the distributive lattice of bands.
\end{notation}

\begin{proposition} \label{p:zaanen_theorem_30_5_ii}
Let $E$ be a real (or complex) vector lattice with sufficiently many projections. If $B_\alpha \uparrow B$ holds, then $\bP_\alpha (u) \uparrow \bP (u)$ for all $u \in E^+$.
\end{proposition}

We introduce a band decomposition of $E$ that will allow us to use classical arguments that rely on the law of trichotomy for the real numbers. This aim motivates the following notation.

\begin{notation}
Let $E$ be a real (or complex) vector lattice with the projection property, and consider $x, y \in E_\rho$. We write the band generated by $(y - x)^+$ as $B_{x < y}$, and write $B_{x \leq y}$ for $B_{y < x}^d$. Finally, $B_{x = y}$ is defined to be the band $B_{x \leq y} \cap B_{y \leq x}$. The corresponding band projections will be written as $\bP_{x < y}$, $\bP_{x \leq y}$ and $\bP_{x=y}$, respectively.
\end{notation}

\begin{example}
Consider $E = \bR^{[0, 1]}$ and $x,y\in E$. Then the bands $B_{x < y}$, $B_{x = y}$, and $B_{x > y}$ correspond to the subsets of $[0, 1]$ where these inequalities hold pointwise. The following graphic illustrates this fact.

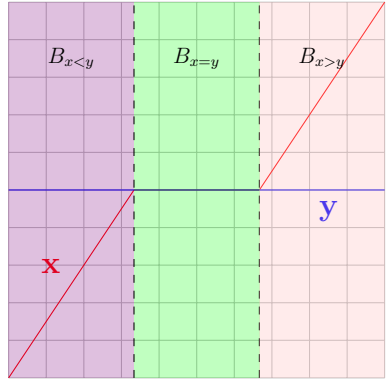
\begin{center}
\begin{tikzpicture}[scale = 5]

\draw[lightgray, step = 0.1] (0, 0) grid (1, 1);

\draw[red] (0, 0) -- (1/3, 1/2) node[midway, above left] {$\mathbf{x}$};
\draw[red] (1/3, 1/2) -- (2/3, 1/2);
\draw[red] (2/3, 1/2) -- (1, 1);

\draw[blue] (0, 1/2) -- (1, 1/2) node[midway, right = 50pt, below] {$\mathbf{y}$};

\fill[violet, opacity=0.25] (0,0) -- (1/3, 0) -- (1/3, 1) -- (0, 1) -- cycle;
\node[scale = 0.7] at (1/6, 0.85) {$B_{x < y}$};

\fill[green, opacity = 0.25] (1/3, 0) -- (2/3, 0) -- (2/3, 1) --(1/3, 1) -- cycle;
\node[scale = 0.7] at (1/2, 0.85) {$B_{x = y}$};
\draw[dashed, black] (1/3, 0) -- (1/3, 1);
\draw[dashed, black] (2/3, 0) -- (2/3, 1);

\fill[pink, opacity = 0.3] (2/3, 0) -- (1, 0) -- (1, 1) -- (2/3, 1) -- cycle;
\node[scale = 0.7] at (5/6, 0.85) {$B_{x > y}$};

\end{tikzpicture}
\end{center}

\end{example}

\begin{proposition} \label{p:bands_decomposition}
Let $E$ be a real (or complex) vector lattice with the projection property, and put $x, y \in E_\rho$. Then we get the following band decompositions of $E$:
\[ E = B_{x < y} \oplus B_{y \leq x} = B_{x < y} \oplus B_{y < x} \oplus B_{x = y}. \]
\end{proposition}
\begin{proof}
Since $B_{y \leq x} = B_{x < y}^d$, we know that $E = B_{x < y} \oplus B_{y \leq x}$. For the second decomposition, we show that $B_{x \leq y} = B_{x < y} \oplus B_{x = y}$. To this end, we know that $B_{x = y} \subseteq B_{y \leq x} = B_{x < y}^d$, so $B_{x < y} \cap B_{x = y} = \{0\}$. Since $(y - x)^+ \perp (x - y)^+$, $B_{x < y} = B_{(y-x)^+} \subseteq B_{(x-y)^+}^d = B_{x \leq y}$. Putting the last two facts together gives us
\begin{align*}
B_{x < y}
 \oplus B_{x = y} & = B_{x < y} + (B_{x \leq y} \cap B_{y \leq x}) \\
& = (B_{x < y} + B_{x \leq y}) \cap (B_{x < y} + B_{y \leq x}) \\
& = B_{x \leq y} \cap E  \\ &= B_{x \leq y}. \hfill \qedhere
\end{align*}
\end{proof}

\begin{lemma}
Let $E$ be a real (or complex) vector lattice with the projection property, and assume that $x, y \in E_\rho$. Then
\[ \bP_{x < y} (y - x) = \bP_{x \leq y} (y - x) = (y - x)^+ \text{ and } \bP_{x = y} (y - x) = 0. \] 
\end{lemma}
\begin{proof}
Since we have that $(y - x)^+ \in B_{x < y}$ and $(y - x)^- \in B_{y \leq x}$, it follows that $\bP_{x < y} (y - x) = (y - x)^+$ and $\bP_{y \leq x} (y - x) = -(y-x)^-$. By interchanging $x$ and $y$, we obtain $\bP_{x \leq y} (y - x) = (x - y)^- = (y - x)^+$. Now,
\[ \bP_{x = y}(y - x) = \bP_{y \leq x} \bP_{x \leq y} (y - x) = \bP_{y \leq x} (y - x)^+ \geq 0. \]
Similarly, $\bP_{x = y} (x - y) \geq 0$. Hence, $\bP_{x = y}(x - y) = 0$. 
\end{proof}

\begin{proposition} \label{p:bands_largest_characterization}
Let $E$ be a real (or complex) vector lattice with the projection property, and consider $x, y \in E_\rho$. Then
\begin{itemize}
\item[$(i)$] $B_{x \leq y}$ is the largest band for which the corresponding band projection $\bP$ satisfies $\bP (x) \leq \bP (y)$,
\item[$(ii)$] $B_{x = y}$ is the largest band for which the corresponding band projection $\bP$ satisfies $\bP (x) = \bP (y)$, and
\item[$(iii)$] $B_{x < y}$ is the largest band for which $\bP (y - x)$ is a weak order unit.
\end{itemize}
\end{proposition}

\begin{proof}
It is immediate from the previous lemma that the above bands have the claimed properties. It remains to be shown that they are the largest such bands. To this end, let $B$ be a band with band projection $\bP$.

(i) Suppose $\bP (x) \leq \bP (y)$. Then $\bP ( (x - y)^+) = \bP(x - y) - \bP( (x - y)^-) \leq 0$, but we must have that $\bP ( (x - y)^+) \geq 0$. Hence, $\bP( (x - y)^+ ) = 0$. Since $(x - y)^+$ is a weak order unit of $B_{y < x}$, this means $B \perp B_{y < x}$. We conclude that  $B \subseteq B_{x \leq y}$.

(ii) Suppose $\bP(x) = \bP(y)$. By (i), we have that $B \subseteq B_{x \leq y} \cap B_{y \leq x} = B_{x = y}$.

(iii) Suppose $\bP(y - x)$ is a weak order unit of $B$. Then
\[ 0 \leq \bP(y - x) = \bP(y - x) \vee 0 = \bP( (y-x) \vee 0) = \bP( (y-x)^+) \leq (y-x)^+.\]
Thus, $\bP(y - x) \in B_{x < y}$ and hence $B \subseteq B_{x < y}$.
\end{proof}

\subsection{An $\eps$-$\delta$ reformulation of order continuity}

The notation $f \colon \dom(f) \to E$ refers to a function whose domain $\dom(f)$ is a subset of $E$. We define what it means for such functions to be order continuous in terms of order convergence, and reformulate the definition to resemble the classical $\eps$-$\delta$ definition.

\begin{notation}
Given a real (or complex) vector lattice $E$, the notation $A \searrow 0$ means $A \subseteq E_\rho$ and $\inf A = 0$. If additionally $A$ is downwards directed, we write $A \downarrow 0$.

% Analogous definitions hold for the arrows $\uparrow$ and $\nearrow$. 
\end{notation}

\begin{definition}\label{d:epsilon delta}
Let $E$ be a real (or complex) vector lattice. A net $(x_\alpha)_\alpha$ in $E$ is said to \emph{converge in order} to $x \in E$, written $x_\alpha \to x$, if there exists an $\Eps \searrow 0$ such that for every $\eps \in \Eps$ there is an $\alpha_0$ satisfying
\[ \alpha \geq \alpha_0 \implies \abs{x_\alpha - x} \leq \eps. \]
\end{definition}

This convergence is the only one on $E$ that will be considered in this paper.

\begin{remark}\label{r:eps directed}
    In \Cref{d:epsilon delta}, the set $\Eps$ can be replaced by the set of finite infima of its elements, making it downwards directed, but without changing the infimum.
\end{remark}

\begin{definition}
Let $E$ be a real (or complex) vector lattice. A function $f\colon \dom(f) \to E$ is called \emph{order continuous at $c \in \dom(f)$} if whenever $x_\alpha \to c$ in $\dom(f)$, we have $f(x_\alpha) \to f(c)$. The function $f$ is called \emph{order continuous} if it is order continuous at every point in $\dom(f)$.

% We say that $f$ is \emph{uniformly (order) continuous} if whenever $x_\alpha - y_\alpha \to 0$ in $\dom(f)$, we have $f(x_\alpha) - f(y_\alpha) \to 0$.
\end{definition}

In order to make the similarity with classical theory in this paper explicit, the above definition is reformulated to resemble the classical $\eps$-$\delta$ definition.

% We provide the proof for \Cref{p:eps_delta_continuity}, and a similar argument can be used for \Cref{p:eps_delta_uniform_continuity}

% Continuity and differentiability can be reformulated to resemble the classical $(\eps, \delta)$-definitions. \textcolor{pink}{ \st{ Proposition ... is proven below, and similar arguments can be used for the other reformulations.} We reformulate these classical definitions in the following propositions. We provide the proof for \Cref{p:super_diff_eps_delta} and similar arguments can be used for the others.}

\begin{proposition} \label{p:eps_delta_continuity}
Let $E$ be a real (or complex) vector lattice. A function $f \colon \dom(f) \to E$ is order continuous at $c \in \dom(f)$ if and only if for every $\Delta \downarrow 0$ there exists an $\Eps \searrow 0$ with the property that for all $\eps \in \Eps$ there exists a $\delta \in \Delta$ satisfying
\[ x \in \dom(f) \text{ and } \abs{x - c} \leq \delta \implies \abs{f(x) - f(c)} \leq \eps. \]
\end{proposition}

\begin{proof}
Suppose $f$ is order continuous at $c$. Let $\Delta \downarrow 0$. The set 
\[ A \defeq \{(x, \delta) \colon x \in \dom(f), \delta \in \Delta, \text{ and } \abs{x - c} \leq \delta\} \]
is directed by $(x_1, \delta_1) \leq (x_2, \delta_2)$ if and only if $\delta_2 \leq \delta_1$. Define, for any $\alpha = (x, \delta)$ in $A$, $x_\alpha \defeq x$ so that $x_\alpha \to c$ in $\dom(f)$. By order continuity, $f(x_\alpha) \to f(c)$, so there exists an $\Eps \searrow 0$ with the property that for every $\eps \in \Eps$ there exists an $\alpha_0 = (x_0, \delta_0)$ such that for $\alpha \geq \alpha_0$, we have $\abs{f(x_\alpha) - f(c)} \leq \eps$. If $x \in \dom(f)$ with $\abs{x - c} \leq \delta_0$, then $(x, \delta_0) \geq (x_0, \delta_0) = \alpha_0$, and so we have $\abs{f(x) - f(c)} \leq \eps$.

Conversely, suppose $f$ satisfies the condition in the proposition. Suppose that $x_\alpha \to c$ in $\dom(f)$. Then there exists a $\Delta \downarrow 0$ such that for all $\delta \in \Delta$ there exists an $\alpha_0$ for which $\abs{x_\alpha - c} \leq \delta$ for all $\alpha \geq \alpha_0$. For this $\Delta$, there exists an $\Eps \searrow 0$ that satisfies our assumption. Let $\eps \in \Eps$ with corresponding $\delta$ and $\alpha_0$. Then, for all $\alpha \geq \alpha_0$, we have $\abs{x_\alpha - c} \leq \delta$ so $\abs{f(x_\alpha) - f(c)} \leq \eps$. Therefore, $f(x_\alpha) \to f(c)$.
\end{proof}

% \begin{proposition} \label{p:eps_delta_uniform_continuity}
% A function $f \colon \dom(f) \to E$ is uniformly continuous if and only if for every $\Delta \downarrow 0$ there exists an $\Eps \searrow 0$ with the property that for all $\eps \in \Eps$ there exists a $\delta \in \Delta$ satisfying
% \[ x, y \in \dom(f) \text{ and } \abs{x - y} \leq \delta \implies \abs{f(x) - f(y)} \leq \eps. \]
% \end{proposition}

In this characterisation, it is immaterial whether $\Eps$ is required to be directed or not (as explained in \Cref{r:eps directed}). However, it is essential that we only consider downwards directed $\Delta$. The following example shows how our reformulation of order continuity fails for $\Delta \searrow 0$.

\begin{example}

Consider $f \colon \bR^2 \to \bR^2$ defined by $f(x, y) := (x + y, 0)$. It is clear that $f$ is order continuous at $(0, 0)$. Consider $\Delta \defeq \{(1, 0), (0, 1)\} \searrow 0$. Since $f(0, 0) = (0, 0)$, $f(1, 0) = (1, 0)$, and $f(0, 1) = (1, 0)$, there can be no $\Eps \searrow 0$ that satisfies our characterisation of order continuity for this set $\Delta$, since $\Eps$ would have to be bounded below by $(1, 0)$.
\end{example}

\section{Order differentiability}

We define what it means for a function to be differentiable on subsets of $E$, as was first done in \cite{differentiation}, for complex $\Phi$-algebras. Note that this definition can also be used for real $\Phi$-algebras. 

\begin{definition} \label{d:ll}
Let $E$ be a Dedekind complete real (or complex) $\Phi$-algebra, and suppose that $x, y \in E_\rho$. We write $x \ll y$ (or $y \gg x$) to mean $y-x$ is a weak order unit of $E$. In particular, $r \gg 0$ means that $r$ is a weak order unit.
\end{definition}

\begin{definition}
Let $E$ be a Dedekind complete real (or complex) $\Phi$-algebra. For $c \in E$ and $r \gg 0$, define
\[ N(c, r) \defeq \{z \in E \colon \abs{z - c} \ll r \}, \quad \text{and} \quad \barN(c, r)  \defeq \{y \in E \colon \abs{z - c} \leq r \}. \]
If $E$ is real and $a, b \in E$, we define
\begin{align*}
(a, b) & \defeq \{x \in E \colon a \ll x \ll b\} \quad \text{ and}\quad [a, b] \defeq \{x \in E \colon a \leq x \leq b \}.
\end{align*}
In this case, $N(c, r) = (c - r, c + r)$ and $\barN(c, r) = [c - r, c + r]$.
\end{definition}

There are two natural notions of differentiability introduced in \cite[Definitions 3.5 and 4.1]{differentiation}. Using similar arguments as in the proof of \Cref{p:eps_delta_continuity}, these can be given equivalent $\epsilon$-$\delta$ definitions.

\begin{definition} \label{d:diff}
Let $E$ be a Dedekind complete real (or complex) $\Phi$-algebra. A function $f \colon \dom(f) \to E$ is called \emph{order differentiable} at $c \in \dom(f)$ with derivative $f'(c)$ if there exists an $r \gg 0$ such that $N(c, r) \subseteq \dom(f)$, and, for every $\Delta \downarrow 0$, there exists an $\Eps \searrow 0$ with the property that for all $\eps \in \Eps$ there exists a $\delta \in \Delta$ satisfying
\[ z \in N(c, r) \text{ and } \abs{z - c} \leq \delta \implies \abs{f(z) - f(c) - (z - c)f'(c)} \leq \abs{z - c} \eps. \]
\end{definition}

\begin{definition}
Let $E$ be a Dedekind complete real (or complex) $\Phi$-algebra. A function $f \colon \dom(f) \to E$ is called \emph{super order differentiable} at $c \in \dom(f)$ with derivative $f'(c)$ if there exists an $r \gg 0$ such that $N(c, r) \subseteq \dom(f)$, and, for every $\Delta \downarrow 0$, there exists an $\Eps \searrow 0$ with the property that for all $\eps \in \Eps$ there exists a $\delta \in \Delta$ satisfying
\[ z \in \dom(f) \text{ and } \abs{z - c} \leq \delta \implies \abs{f(z) - f(c) - (z - c)f'(c)} \leq \abs{z - c} \eps. \]
\end{definition}

The above definitions differ in whether the implication is required to hold only for $z \in N(c, r)$ or for all $z \in \dom(f)$. These definitions are not equivalent because, for a given $r \gg 0$, it is possible to have a net $(z_\alpha)_\alpha$ in $\dom(f)$ converging to $c$ that never enters $N(c, r)$. Therefore, to determine whether a function is order differentiable at $c$, one can ignore how $f$ acts on certain nets converging to $c$. However, all nets converging to $c$ must be considered to determine super order differentiability.

\begin{remark}
In \cite[Notation 3.2]{differentiation}, $x \ll y$ is defined to mean that $y - x$ is a positive invertible element. Every positive invertible element is a weak order unit \cite[Theorem~142.2 (ii)]{zaanen2}, and the converse holds when $E$ is universally complete but not in general \cite[Remark~3.3]{series}. For example, $(\frac{1}{n})_{n \in \bN}$ is a weak order unit in $\ell^\infty(\bN)$ but is not invertible. Changing this definition is necessary for the techniques used in Section $5$, where it is needed that $\bP_{x < y}(x) \ll \bP_{x < y}(y)$, when seen as elements of $B_{x < y}$.

The results in \cite{differentiation} remain true when the positive invertible element $r$ is replaced by a weak order unit, except for the argument showing that the derivative is unique in \cite[Remark 3.6]{differentiation}. A modified proof is given in \Cref{p:derivative_unique}.
\end{remark}

For $m \in \bN$ and a weak order unit $r$, we have that $\bP_{\frac{1}{m}e \leq r}(r)$ must be invertible in $B_{\frac{1}{m}e\leq r}$ by \cite[Theorem~1.11]{depagter}, as it is larger than the invertible element $\bP_{\frac{1}{m} e \leq r}(\frac{1}{m}e)$. The following lemma therefore allows us to still use invertibility arguments for weak order units.

\begin{lemma} \label{l:projection_weak_unit_invertible}
Let $E$ be a Dedekind complete real (or complex) $\Phi$-algebra, and assume that $r$ is a weak order unit of $E$. Then $B_{\frac{1}{m}e \leq r} \uparrow E$ as $m \to \infty$ and hence $\bP_{\frac{1}{m}e \leq r} (u) \uparrow u$ as $m \to \infty$ for all $u \in E^+$.
\end{lemma}

\begin{proof}
Let $B_0$ be a band such that $B_{\frac{1}{m}e \leq r} \subseteq B_0$ for all $m \in \bN$. Then, for any $m \in \bN$, $B_0^d \subseteq B_{\frac{1}{m}e > r}$ and so $\bP_0^d (r) \leq \bP_0^d (\frac{1}{m}e)$. Therefore, $\bP_0^d (r) = \{ 0\}$, so $B_0^d = \{0\}$ and hence $B_0 = E$. Now \Cref{p:zaanen_theorem_30_5_ii} completes the proof.
\end{proof}

\begin{proposition} \label{p:derivative_unique}
Let $E$ be a Dedekind complete real (or complex) $\Phi$-algebra. If $f \colon \dom(f) \to E$ is order differentiable at $c \in \dom(f)$, there is a unique value of $f'(c)$ that satisfies \Cref{d:diff}.
\end{proposition}

\begin{proof}
Let $r$ be as in \Cref{d:diff}. For $\Delta \defeq \{\frac{1}{n} r\colon n \in \bN \}$, let $\Eps$ be as in \Cref{d:diff}. Fix $m \in \bN$. Since $\frac{1}{m}e$ is invertible, $\bP_{\frac{1}{m}e \leq r}(r)$ is invertible in $B_{\frac{1}{m}e \leq r}$. Let $s$ denote its inverse.

For a given $\eps \in \Eps$ with corresponding $\delta = \frac{1}{M} r$, we have that for all $n \geq M$, $z_n \defeq c + \frac{1}{n + 1} r$ satisfies $z_n \in N(c, r)$ and $\abs{z_n - c} \leq \delta$, so that
\[ \abs*{f(z_n) - f(c) - {\textstyle \frac{1}{n + 1}} r f'(c)} \leq {\textstyle \frac{1}{n + 1}} r \eps, \]
and hence
\[ \abs*{(n+1) s \bP_{\frac{1}{m}e \leq r} (f(z_n) - f(c)) - \bP_{\frac{1}{m}e \leq r} (f'(c))} \leq \bP_{\frac{1}{m}e \leq r} (\eps). \]
Therefore, $(n+1) s \bP_{\frac{1}{m}e \leq r} (f(z_n) - f(c)) \to \bP_{\frac{1}{m}e \leq r} (f'(c))$. Hence the derivative is unique in $B_{\frac{1}{m}e \leq r}$  for all $m \in \bN$. By \Cref{l:projection_weak_unit_invertible}, it is uniquely determined in $E$.
\end{proof}

We recall the sum and product rules for order differentiable and super order differentiable functions from \cite[Theorems 3.15 and 4.3]{differentiation}, where the reader may also find chain and quotient rules.

\begin{proposition} \label{p:properties_derivative}
Let $E$ be a Dedekind complete real (or complex) $\Phi$-algebra, and assume that $f\colon \dom(f) \to E$ and $g\colon \dom(f) \to E$, and let $f$ and $g$ be (super) order differentiable at $c \in \dom(f) \cap \dom(g)$. Then,
\begin{itemize}
    \item[$(i)$] $f + g$ is (super) order differentiable at $c$ with $(f+g)'(c) = f'(c) + g'(c)$, and
    \item[$(ii)$] $fg$ is (super) order differentiable at $c$ with $(fg)'(c) = f'(c) g(c) + f(c) g'(c)$.
\end{itemize}
\end{proposition}

\begin{example}
For any $E$, it is routine to verify that $f(x) := x$ and $g(x) := c$, for some $c \in E$, are both super order differentiable everywhere with $f'(x) = e$ and $g'(x) = 0$. It follows from \Cref{p:properties_derivative} that any polynomial on $E$ is super order differentiable with the expected derivative.
\end{example}

The following proposition shows that order differentiability, as opposed to super order differentiability, is a local property.

\begin{proposition} \label{p:differentiation_local}
Let $E$ be a Dedekind complete real (or complex) $\Phi$-algebra, and additionally suppose that $f\colon \dom(f) \to E$ is order differentiable at $c \in \dom(f)$. If $g\colon \dom(g) \to E$ is defined on, and equal to, $f$ on $N(c, r)$ for some $r \gg 0$, then $g$ is order differentiable at $c$ with $g'(c) = f'(c)$.
\end{proposition}
\begin{proof}
Since $f$ is order differentiable at $c$, there exists an $r_1 \gg 0$ such that $N(c, r_1) \subseteq \dom(f)$, and for every $\Delta \downarrow 0$, there exists an $\Eps \searrow 0$ with the property that for all $\eps \in \Eps$ there exists a $\delta \in \Delta$ satisfying
\[ z \in N(c, r) \text{ and } \abs{z - c} \leq \delta \implies \abs{f(z) - f(c) - (z - c)f'(c)} \leq \abs{z - c} \eps. \]
Define $r_2 \defeq r \wedge r_1$, then $N(c, r_2) \subseteq \dom(g)$. Let $\eps, \delta$ be as above. Then, for all $z \in N(c, r_2)$ with $\abs{z - c} \leq \delta$,
\[ \abs{g(z) - g(c) - (z - c) f'(c)} = \abs{f(z) - f(c) - (z - c)f'(c)} \leq \abs{z - c} \eps. \]
Therefore, $g$ is order differentiable at $c$ with $g'(c) = f'(c)$.
\end{proof}

\section{Locally band preserving functions}

This section is dedicated to the notion of locally band preserving functions. The main reason being that order continuous functions on order intervals, in general, do not inherit the properties real-valued functions on real intervals posses. We start with \Cref{e:cont_bounded_fail}, where an example is given of an unbounded order continuous function on a closed order interval. \Cref{e:IVT_EVT_fail} shows that the Intermediate Value Theorem and the Extreme Value Theorem can fail for order continuous and order bounded functions defined on a closed order interval. These results are recovered for locally band preserving, order continuous functions in Section 5.

\begin{example}\label{e:cont_bounded_fail}
   Consider the Dedekind complete $\Phi$-algebra $L^0[0,1]$ of Le-besgue measurable functions on $[0,1]$, with $e$ being the constant one function. Then $[0,e] = L^1[0,1] \cap [0,e]$. Since $L^1[0,1]$ is a Banach lattice with order continuous norm $\|\cdot\|_1$, for this norm the order interval $[0,e]$ is a metric space that is not compact. It follows that there must exist a continuous function $f \colon [0,e] \to \bR$ that is unbounded. Indeed, for a sequence $(x_n)_{n \in \bN}$ in $[0,e]$ that does not have a norm convergent subsequence, we can define the function $g \colon \{x_n\colon n\in \bN\} \to \bR$ by $g(x_n) := n$. Since $\{x_n \colon n \in \bN\}$ is a closed and discrete set in $[0,e]$, and $g$ is continuous, the Tietze extension theorem implies that $g$ can be extended continuously to all of $[0,e]$, where it is necessarily unbounded as well. Note that if $x_\alpha \to x$ in order in $[0,e]$, then there is a set $\Eps \downarrow 0$ such that for all $\eps \in \Eps$ there is an $\alpha_0$ such that $|x_\alpha - x| \le \eps$ whenever $\alpha \ge \alpha_0$. It follows from the fact that $|x_\alpha - x| \le 2e$ for all $\alpha$ that we may assume without loss of generality that $\Eps \subseteq L^1[0,1]$. Hence $\|x_\alpha - x\|_1 \le \|\eps\|_1$, and since $\inf \{\|\eps\|_1 \colon \eps \in \Eps\} = 0$ as $\|\cdot\|_1$ is order continuous, this net converges in norm. Thus $|f(x) - f(x_\alpha)| \to 0$, and so $f$ is order continuous. In conclusion, the function $h \colon [0,e] \to L^0[0,1]$ defined by $h(x) := f(x)e$ provides an example of an order continuous function which is unbounded.       
\end{example}

\begin{example} \label{e:IVT_EVT_fail}
Let $f, g \colon [(0, 0), (1, 1)] \to \bR^2$ be given by $f(x, y) := (x, x^2)$ and $g(x, y) := (x, 1 - x)$. Then $f$ and $g$ are both order continuous. Since $f(0, 0) = (0, 0)$ and $f(1, 1) = (1, 1)$ but $(\frac12, \frac12) \not\in \range(f)$, the Intermediate Value Theorem fails for $f$. The Extreme Value Theorem fails for $g$, even though $g$ is order bounded, since
\[ \sup_{(x, y) \in [(0, 0), (1, 1)]} g(x, y) = (1, 1),\]
but this supremum is not attained.
\end{example}

The shortcomings of Example~\ref{e:IVT_EVT_fail} can be rectified by considering only locally band preserving functions. These functions were first introduced for Archimdedean real vector lattices in \cite{Wickstead-Zafer}. 

\begin{definition}
Let $E$ be a Dedekind complete real (or complex) $\Phi$-algebra, and consider $f \colon \dom(f) \to E$. We say that $f$ is \emph{locally band preserving} if for all $x, y \in \dom(f)$ and $z \in E$ we have that 
\[
|x-y| \wedge |z| = 0 \Longrightarrow |f(x) - f(y)| \wedge |z| = 0.
\]
\end{definition}

We next provide a characterisation of locally band preserving functions using band projections, which we will freely use throughout the remainder of this paper. 

\begin{proposition}
    Let $E$ be a Dedekind complete real (or complex) $\Phi$-algebra. A function $f \colon \dom(f) \to E$ is locally band preserving if and only if for every band projection $\bP$ and for all $x,y \in \dom(f)$, we have that 
    \[
    \bP(x) = \bP(y) \Longrightarrow \bP(f(x)) = \bP(f(y)).
    \]
\end{proposition}

\begin{proof}
    Suppose $f$ is locally band preserving, and let $B$ be a band with corresponding band projection $\bP$. If $x, y \in \dom(f)$ are such that $\bP(x) = \bP(y)$, we have if $z \in B$, then $|x - y| \wedge |z| = 0$, so $|f(x) - f(y)| \wedge |z| = 0$. Hence $|\bP(f(x) - \bP(f(y))| = \bP(|f(x) - f(y)|) = 0$, and, therefore, $\bP (f(x)) = \bP(f(y))$. 

    Conversely, suppose that for all $x,y \in \dom(f)$ and for every band projection $\bP$, we have that $\bP(x) = \bP(y)$ implies $\bP(f(x)) = \bP(f(y))$. Let $z \in E$ be such that $|x-y| \wedge |z| = 0$. Then for the band $B_z$ with corresponding band projection $\bP_z$, it follows that $\bP_z(x) = \bP_z(y)$, so $\bP_z (f(x)) = \bP_z (f(y))$. Hence $|f(x) - f(y)| \wedge |z| = 0$.
\end{proof}

\begin{remark}
    Note that a locally band preserving function $f \colon E \to E$ is band preserving if $f(0) = 0$. Indeed, for any band $B$ in $E$ and $x \in B$, we have that $\bP^d(x) = \bP^d(0) = 0$, and thus $\bP^d(f(x)) = \bP^d(f(0))$, and hence we have $f(x) \in B$. Furthermore, it is readily checked that a linear function $f \colon E \to E$ is locally band preserving if and only if it is band preserving. Hence an order bounded linear function $f \colon E \to E$ is locally band preserving if and only if it is an orthomorphism.
\end{remark}

\begin{example}
In $\bR$, all functions are locally band preserving, as the only band projections are the zero and identity mappings. A function in $\ell^\infty$ or $\bR^\bN$ is locally band preserving if and only if it is defined coordinatewise. Thus, neither of the functions in \Cref{e:IVT_EVT_fail} is locally band preserving. 
\end{example}

However, if we view these $\Phi$-algebras as $\ell^\infty \cong C(\beta \bN)$ and $\bR^\bN \cong C^\infty(\beta\bN)$, then locally band preserving functions need not be defined pointwise, even if they are order continuous.

\begin{example}\label{E:pointwise}
For each $n \in \bN$, consider the continuous function $k_n \colon \bR \to \bR$ such that $k_n(t) = 0$ for $t \leq \frac{1}{2n}$, $k_n(t) = 1$ for $t \geq \frac{1}{n}$, and $k_n$ is linear on $[\frac{1}{2n}, \frac{1}{n}]$.

Define $f\colon \ell^\infty \to \ell^\infty$ by $f((x(n))_{n \in \bN}) := (k_n(x(n))))_{n \in \bN}$. It is clear that $f$ is locally band preserving since it is defined coordinatewise. To show that $f$ is order continuous, let $x_\alpha \to x$ in $\ell^\infty$. Then, for each $n \in \bN$, we have that $x_\alpha(n) \to x(n)$, and hence $(f(x_\alpha))(n) = k_n(x_\alpha(n)) \to k_n(x(n)) = (f(x))(n)$, since $k_n$ is continuous. Thus, $f(x_\alpha)$ converges to $f(x)$ in each coordinate. Since $f(x_\alpha)$ is bounded between $0$ and $e$, coordinatewise convergence implies order convergence, and so $f(x_\alpha) \to f(x)$.

Identifying $\ell^\infty$ with $C(\beta \bN)$,  we show that even though $f$ is locally band preserving, it does not act pointwise on $\beta \bN$. To this end, let $x := 0$ and $y$ be the extension of $(\frac{1}{n})_{n \in \bN}$ to $\beta \bN$. For every $t \in \beta \bN \setminus \bN$, $x(t) = 0 = y(t)$. However, $f(x) = 0$ and $f(y) = e$, so $f$ does not act pointwise at any point $t \in \beta \bN \setminus \bN$.
\end{example}

The following lemma is needed for the proof of \Cref{t:differentiable_implies_lbp}.

\begin{lemma} \label{l:super_differentiable_locally_lbp}
Let $E$ be a Dedekind complete real (or complex) $\Phi$-algebra, and assume that $f \colon \dom(f) \to E$ is super order differentiable on its entire domain. Then, for all $\delta \in E^+$, every $x \in \dom(f)$, and all band projections $\bP$, there exists a $k \in \bR^+\setminus\{0\}$ satisfying
\[ y \in \dom(f) \text{ and } \abs{y - x} \leq  k \delta \text{ and } \bP (x) = \bP (y) \implies \bP (f(y)) = \bP(f(x)). \] 
\end{lemma}
\begin{proof}
Let $\delta \in E^+$, $x \in \dom(f)$, and $\bP$ be any band projection. Define 
\[
\Delta \defeq \{ k \delta \colon k \in \bR^+\setminus\{0\} \}. 
\]
Then there exists an $\Eps \searrow 0$ such that, for all $\eps \in \Eps$, there is a $k_\eps \in \bR^+\setminus\{0\}$ for which
\[ y \in \dom(f) \text{ and } \abs{y - x} \leq k_\eps \delta \implies \abs{f(y) - f(x) - (y - x)f'(x)} \leq \abs{y - x} \eps. \]
Fix any $\eps \in \Eps$. Then, for all $y \in \dom(f)$ satisfying $\abs{y - x} \leq k_\eps \delta$ and $\bP(x) = \bP(y)$, we apply $\bP$ to both sides of the above inequality to show that
\[ \bP \left( \abs{f(y) - f(x)} \right) = \bP \left( \abs{f(y) - f(x) - (y - x) f'(x)} \right) \leq \bP \left(\abs{y - x} \eps \right) = 0.\] 
Therefore, $\bP(f(y)) = \bP(f(x))$.
\end{proof}

It will now be shown that any super order differentiable function 
\[
f\colon N(c, r) \to E 
\]
is necessarily locally band preserving, which makes locally band preserving functions a natural class of functions to study in this context.

\begin{theorem} \label{t:differentiable_implies_lbp}
Let $E$ be a Dedekind complete real (or complex) $\Phi$-algebra, and consider $c,r\in E$ with $r\gg 0$. If $f \colon N(c, r) \to E$ is super order differentiable, then $f$ is locally band preserving.
\end{theorem}
\begin{proof}
Let $I \defeq \{t \in \bR \colon 0 \leq t \leq 1 \}$. Suppose $x, y \in N(c, r)$ with $\bP(x) = \bP(y)$ for some band projection $\bP$. Define $x_t \defeq (1 - t)x + ty$ for all $t \in I$, and let $L[x, y] \defeq \{x_t \colon t \in I \}$. Then $L[x, y] \subseteq N(c, r)$, and for all $s,t \in I$, 
\begin{align*}
\bP(x_t) = (1 - t) \bP(x) + t \bP(y) = \bP(x),\ \text{ and}\ \ \ \abs{x_s - x_t}  = \abs{s - t} \abs{x - y}.
\end{align*}

Define $g \colon I\to E$ by $g(t) := \bP(f(x_t))$. Let $\delta := \abs{x - y}$ and $t \in I$. By \Cref{l:super_differentiable_locally_lbp}, there exists a $k_t \in \bR^+\setminus\{0\}$ such that
\[ z \in N(c, r) \text{ and } \abs{z - x_t} \leq k_t \delta \text{ and } \bP(z) = \bP(x_t) \implies \bP(f(z)) = \bP(f(x_t)). \]

In particular, if $s \in I$ and $\abs{s - t} \leq k_t$, then $\abs{x_s - x_t} = \abs{s - t} \abs{x - y} \leq k_t \delta$, so $g(s) = \bP (f(x_s)) = \bP(f(x_t)) = g(t)$. Thus, $g$ is constant on 
\[
I_t \defeq \{ s \in I \colon \abs{s - t} \leq k_t \} 
\]
for all $t \in I$.

By compactness, $I \subseteq \bigcup_{i = 1}^n I_{t_i}$ for some $t_1, \dots, t_n \in I$. Then there exist points $s_1, \dots, s_m \in I$ such that $s_1 = 0$, $s_m = 1$, and consecutive points lie in the same interval $I_{t_i}$ for some $i$. This shows that
\[ \bP(f(x)) = g(s_1) = g(s_2) = \dots = g(s_m) = \bP(f(y)). \hfill \qedhere\]
\end{proof}

\begin{corollary} \label{l:differentiable_implies_lbp}
Let $E$ be a Dedekind complete real (or complex) $\Phi$-algebra, and consider $c,r\in E$ with $r\gg 0$. If $f \colon \barN(c, r) \to E$ is order continuous on $\barN(c, r)$ and super order differentiable on $N(c, r)$, then $f$ is locally band preserving.
\end{corollary}

\begin{proof}
Let $x, y \in \barN(c, r)$ with $\bP(x) = \bP(y)$. Define $x_n \defeq c + (1 - \frac{1}{n}) (x - c)$ and $y_n \defeq c + (1 - \frac{1}{n}) (y - c)$ for all $n \in \bN$. Then $x_n \to x$, $y_n \to y$, and for all $n \in \bN$, we have that $x_n, y_n \in N(c, r)$ and $\bP(x_n) = \bP(y_n)$, so that $\bP (f(x_n)) = \bP(f(y_n))$ by \Cref{t:differentiable_implies_lbp}. Since $f$ is order continuous,
\[ \bP (f(x)) = \lim_{n \to \infty} \bP(f(x_n)) = \lim_{n \to \infty} \bP(f(y_n)) = \bP(f(y)). \hfill \qedhere\]
\end{proof}

We next record some elementary properties of locally band preserving functions defined on $[a, b]$ that we require in Section 5. The recurring idea is that one can restrict properties of $f$ (such as being order bounded or order continuous) to individual bands.

\begin{proposition} \label{l:restricting_to_bands}
Let $E$ be a Dedekind complete real $\Phi$-algebra. Suppose that $a,b,c,d\in E$ satisfy $a\leq c \leq d \leq b$. Let $f \colon [a, b] \to E$ be locally band preserving. The following statements hold.
\begin{itemize}
\item[$(i)$] If $f$ is order bounded above on $[c, d]$ by $M$, then for all $x \in [a, b]$, $\bP_{c \leq x} \bP_{x \leq d} f(x) \leq \bP_{c \leq x} \bP_{x \leq d} (M)$.
\item[$(ii)$] If $f$ is order bounded below on $[c, d]$ by $m$, then for all $x \in [a, b]$, $\bP_{c \leq x} \bP_{x \leq d} f(x) \geq \bP_{c \leq x} \bP_{x \leq d} (m)$.
\item[$(iii)$] Let $x \in [a, b]$ and $\eps, \delta\in E^+$. If for all $y \in [a, b]$, $\abs{x - y} \leq \delta$ implies $\abs{f(x) - f(y)} \leq \eps$; then, for any band projection $\bP$, $\bP (\abs{x - y}) \leq \bP (\delta)$ implies $\bP (\abs{f(x) - f(y)}) \leq \bP (\varepsilon)$.
\end{itemize}
\end{proposition}
\begin{proof}
To prove (i), let $x \in [a, b]$. Define the band projection $\bP \defeq \bP_{c \leq x} \bP_{x \leq d}$, and let $y \defeq \bP (x) + \bP^d (c)$. Then $\bP(c) \leq \bP(x) \leq \bP(d)$, $\bP(y) = \bP(x)$, and $\bP^d(c) = \bP^d(y) \le \bP^d(d)$. Together, these statements show that $y \in [c, d]$ and hence $f(y) \leq M$. Since $\bP(y) = \bP(x)$, we have $\bP(f(x)) = \bP(f(y)) \leq \bP(M)$. The argument for (ii) is similar. 

We now prove (iii). Let $\eps, \delta\in E^+$ and $x \in [a, b]$ be such that for all $y \in [a, b]$, $\abs{x - y} \leq \delta$ implies $\abs{f(x) - f(y)} \leq \eps$. Let $y \in [a, b]$ and $\bP$ be any band projection such that $\bP (\abs{x - y}) \leq \bP (\delta)$. Define $z \defeq \bP(y) + \bP^d(x)$, then $\abs{x - z} \leq \delta$, so we know that $\abs{f(x) - f(z)} \leq \eps$. Since $\bP(y) = \bP(z)$, we have that $\bP(f(y)) = \bP(f(z))$, and so $\bP \left( \abs{f(x) - f(y)} \right) = \bP \left( \abs{f(x) - f(z)} \right) \leq \bP (\eps)$. 
\end{proof}

% \begin{notation}
% Let $f \colon [a, b] \to E$ be locally band preserving, and let $\bP$ be any band projection. Define $f_\bP: [\bP(a), \bP(b)] \to E$ by $f_\bP(x) = \bP (f(x + \bP^d(a)))$.
% \end{notation}

% In \Cref{s:classical_results}, we will prove that many classical results for continuous function on $[a, b]$ still hold in our setting but only for locally band preserving continuous functions. Classically, the key to many of these proofs is that if $x, y \in \bR$ we know that precisely one of the following holds: $x = y$, $x > y$, or $x < y$. We can mimic this idea by using the decomposition $E = B_{x = y} \oplus B_{x > y} \oplus B_{x < y}$, but to then consider $f$ as a function on each of these bands separately is where we require the locally band preserving assumption.

\section{Classical theorems for continuous and differentiable functions} \label{s:classical_theorems}

In this section, we generalise some classical results for continuous and differentiable real-valued functions to a Dedekind complete real $\Phi$-algebra $E$, but, in light of \Cref{e:cont_bounded_fail} and Example~\ref{e:IVT_EVT_fail}, adding the requirement that the functions are locally band preserving. 

\begin{lemma} \label{l:bounded_subintervals}
Let $E$ be a Dedekind complete real $\Phi$-algebra and consider $a, b, c, d \in E$ with $a \leq b$ and $c \leq d$. Let $f \colon [a \wedge c, b \vee d] \to E$ be locally band preserving, order bounded on $[a, b]$, and order bounded on $[c, d]$. If $[a, b] \cap [c, d] \neq \varnothing$, then $f$ is order bounded on $[a \wedge c, b \vee d]$.
\end{lemma}
\begin{proof}
Suppose $\abs{f(x)} \leq M_1$ for all $x \in [a, b]$ and $\abs{f(x)} \leq M_2$ for all $x \in [c, d]$. Let $z \in [a, b] \cap [c, d]$. For any $x \in [a \wedge c, b \vee d]$, we have that
\[ B_{a > x} \cap B_{c > x} = B_{a \wedge c > x} = \{0\}, \]
and hence $B_{a > x} \subseteq B_{c \leq x}$. Now,
\[ \bP_{x < z} \bP_{a \leq x} (x) \in \bP_{x < z} \bP_{a \leq x} ([a, b]), \]
and $\bP_{x < z} \bP_{a > x}(x) \in \bP_{x < z} \bP_{a > x} ([c, d])$.
Thus, by \Cref{l:restricting_to_bands}(i) and (ii),
\begin{align*}
 \bP_{x < z} (\abs{f(x)})
 & = \bP_{x < z} \bP_{a \leq x} (\abs{f(x)}) + \bP_{x < z} \bP_{a > x} (\abs{f(x)}) \\
 & \leq \bP_{x < z} \bP_{a \leq x} (M_1) + \bP_{x < z} \bP_{a > x} (M_2) \\
 & \leq \bP_{x < z} (M_1 \vee M_2).
\end{align*}
Similarly, $\bP_{x \geq z} (\abs{f(x)}) \leq \bP_{x \geq z} (M_1 \vee M_2)$, and hence $\abs{f(x)} \leq M_1 \vee M_2$.
\end{proof}

\bigskip

The fact that order boundedness extends on connected order intervals from the above lemma, in combination with the property of a function being locally band preserving, provides a boundedness theorem.  

\begin{theorem}[The Boundedness Theorem] \label{t:BT}
Let $E$ be a Dedekind complete real $\Phi$-algebra. Consider $a,b \in E$ such that $a \le b$. If $f \colon [a, b] \to E$ is order continuous and locally band preserving, then $f$ is order bounded.
\end{theorem}
\begin{proof}
Let $S \defeq \{x \in [a, b] \colon f \text{ is order bounded on } [a, x]\}$. Then 
\[
c \defeq \sup(S) \in [a, b]. 
\]
It is clear from \Cref{l:bounded_subintervals} that $S$ is upwards directed.

Consider $\Delta \defeq \{(c - x) \vee \frac{1}{n} (b - c): x \in S, n \in \bN \} \downarrow 0$. Since $f$ is order continuous at $c$, there is some $\eps\in E^+$ and $\delta = (c - x) \vee \frac{1}{n}(b - c)$, for $x \in S$ and $n \in \bN$, such that
\[ y \in [a, b] \text{ and } \abs{y - c} \leq \delta \implies \abs{f(y) - f(c)} \leq \eps \implies \abs{f(y)} \leq \abs{f(c)} + \eps. \]
This shows that $f$ is order bounded on $[x, c]$ and on $[c, c + \frac{1}{n}(b - c)]$, and we already know that $f$ is order bounded on $[a, x]$. By \Cref{l:bounded_subintervals}, $c \in S$ and $c + \frac{1}{n}(b - c) \in S$. Therefore, $c \leq c + \frac{1}{n}(b - c)  \leq c$ which implies that $c = b$.
\end{proof}

We next prove an Intermediate Value Theorem for locally band preserving functions.

\begin{theorem}[The Intermediate Value Theorem] \label{t:IVT}
Let $E$ be a Dedekind complete real $\Phi$-algebra. Consider $a,b \in E$ with $a \le b$. If $f \colon [a, b] \to E$ is order continuous and locally band preserving, and $y \in [f(a) \wedge f(b), f(a) \vee f(b)]$, then there exists $c \in [a, b]$ such that $f(c) = y$.
\end{theorem}
\begin{proof}
We first assume that $y = 0$ and $f(a) \leq f(b)$. Define 
\[
S \defeq \{x \in [a, b] \colon f(x) \leq 0\}. 
\]
Then $S$ is upwards directed since if $x, y \in S$ and $z = x \vee y$, we have that $z = \bP_{x \geq y}(x) + \bP_{x < y}(y)$ from which it follows that 
\[
f(z) = \bP_{x \geq y} (f(z)) + \bP_{x < y} (f(z)) = \bP_{x \geq y} (f(x)) + \bP_{x < y} (f(y)) \leq 0. 
\]
Thus, $S \uparrow c$ for some $c \in [a, b]$ and, by order continuity, $f(c) \leq 0$.

Consider $\Delta \defeq \{ \frac{1}{n} (b - c): n \in \bN \}$. There exists an $\Eps \searrow 0$ such that for every $\eps \in \Eps$ there exists a $\delta \in \Delta$ satisfying
\[ x \in [a, b] \text{ and } \abs{x - c} \leq \delta \implies \abs{f(x) - f(c)} \leq \eps. \]

Let $\eps \in \Eps$, and consider $\delta$ as above. Let $x := \bP_{-f(c) > \eps}(c + \delta) + \bP_{-f(c) > \eps}^d (c)$. Since $\bP_{-f(c) > \eps}^d (x) = \bP_{-f(c) > \eps}^d (c)$ and $f$ is locally band preserving, 
\[ \bP_{-f(c) > \eps}^d (f(x)) = \bP_{-f(c) > \eps} (f(c)) \leq 0. \]
Also, since $\abs{x - c} \leq \delta$,
\[ \bP_{-f(c) > \eps} (f(x) - f(c)) \leq \bP_{-f(c) > \eps}(\eps) \leq \bP_{-f(c) > \eps} (-f(c)), \]
so that $\bP_{-f(c) > \eps} (f(x)) \leq 0$. Therefore,
\[f(x) = \bP_{-f(c) > \eps}^d (f(x)) + \bP_{-f(c) > \eps} (f(x)) \le 0,\]
so that $x \in S$. Hence we obtain
\[ c + \bP_{-f(c) > \eps}(\delta) = x \leq c.\]
Combining this inequality with the fact that $\bP_{-f(c) > \eps}(\delta) \geq 0$, we get that $\bP_{-f(c) > \eps}(\delta) = 0$. Since
\[ \bP_{c = b} (-f(c)) = \bP_{c = b} (-f(b)) \leq 0 \leq \bP_{c = b} (\eps), \]
we have that $B_{c = b} \subseteq B_{-f(c) \leq \eps}$ by \Cref{p:bands_largest_characterization}(i), hence $B_{-f(c) > \eps} \subseteq B_{c < b}$. Thus, since $\delta = \frac{1}{n} (b - c)$ for some $n \in \bN$ and $\bP_{c < b}(b - c)$ is a weak order unit of $B_{c < b}$, we have that $\bP_{-f(c) > \eps}(\delta)$ is a weak order unit of $B_{-f(c) > \eps}$. Hence $B_{-f(c) > \eps} = \{0\}$, so that $-f(c) \leq \eps$. Since $\eps$ was arbitrary, $-f(c) \leq 0$. As we already knew that $f(c) \leq 0$, this means that $f(c) = 0 = y$.

The case where $f(b) \leq f(a)$ is similar. Once these two cases are established, we can find $c, d \in [a, b]$ for which we simultaneously have that $\bP_{f(a) \leq f(b)} (f(c)) = 0$ and $\bP_{f(a) > f(b)} (f(d)) = 0$. Since $f$ is locally band preserving, it then follows that
\[
f(\bP_{f(a) \leq f(b)}(c) + \bP_{f(a) > f(b)}(d)) = 0. 
\]
Finally, if $y \neq 0$, apply the result to $x \mapsto f(x) - y$.
\end{proof}

The following lemma will be used in the proof of the Extreme Value Theorem below. It is a vector lattice analogue of the $\eps$ characterisation of the supremum and infimum in $\bR$.

\begin{lemma} \label{l:wou_inf_sup_bands}
Suppose $E$ is a Dedekind complete real $\Phi$-algebra with $E \neq \{0\}$. Let $r \gg 0$, and let $A \subseteq E$ have infimum $\lambda$ (resp. supremum $\mu$). Then there exist an $a \in A$ such that $B_{a < \lambda + r} \neq \{0\}$ (resp. $B_{a > \mu - r} \neq \{0\}$).
\end{lemma}
\begin{proof}
If $B_{a < \lambda + r} = \{0\}$, then $\lambda + r \leq a$. Thus, if this inequality holds for all $a \in A$, we have that $\lambda + r$ is a lower bound of $A$, and hence $r = 0$, contradicting that $E \neq \{0\}$. Showing that $B_{a > \mu - r} \neq \{0\}$ is similar.
\end{proof}

\begin{theorem}[Extreme Value Theorem] \label{t:EVT}
Let $E$ be a Dedekind complete real $\Phi$-algebra. Consider $a,b \in E$ such that $a \le b$. If $f \colon [a, b] \to E$ is order continuous and locally band preserving, then there exist $c, d \in [a, b]$ such that for all $x \in [a, b]$, $f(c) \leq f(x) \leq f(d)$.
\end{theorem}
\begin{proof}

Assume $E \neq \{0\}$, noting that the result is trivial otherwise. We will first show the existence of $d$. By \Cref{t:BT}, $M(x) \defeq \sup_{y \in [a, x]} f(y)$ exists for all $x \in [a, b]$.

Assume $f(a) \ll M(b)$. Since 
\[
\bP_{a = b}(M(b) - f(a)) = \bP_{a = b} (M(a) - f(a)) = \bP_{a = b} (f(a) - f(a)) =  0, 
\]
$B_{a = b} = \{0\}$, and hence $E = B_{a < b}$. By \Cref{p:bands_largest_characterization}(iii), $a \ll b$. Let 
\[
S := \{x \in [a, b] \colon M(b) - M(x) \text{ is a weak order unit} \}. 
\]
Then $d \defeq \sup S \in [a, b]$. Assume further that $f(d) \ll M(b)$.

Consider $\Delta \defeq \{ \frac{1}{n} (b - a) \colon n \in \bN\}$. Since $f$ is order continuous at $d$, there exists an $\Eps \searrow 0$ such that for every $\eps \in \Eps$ there exists a $\delta \in \Delta$ satisfying
\[ x \in [a, b] \text{ and } \abs{x - d} \leq \delta \implies \abs{f(x) - f(d)} \leq \eps. \]

By \Cref{l:wou_inf_sup_bands}, there exists an $\eps \in \Eps$ (with corresponding $\delta$) such that $B_{\eps < \frac12( M(b) - f(d))} \neq \{0\}$. If we denote the corresponding band projection by $\bP_\eps$, then the Dedekind complete real $\Phi$-algebra $B_{\eps < \frac12( M(b) - f(d))}$ contains $\bP_\eps(S)$ and $\bP_\eps(\delta)$, which is a weak order unit of $B_{\eps < \frac12( M(b) - f(d))}$ since $\delta = \frac{1}{n} (b-a)$ is a weak order unit of $E$. Thus, it follows from \Cref{l:wou_inf_sup_bands} that there exists an $s \in S$ satisfying
\[
B_{\eps < \frac12( M(b) - f(d))} \cap B_{s > d - \delta} = \bP_\eps(B_{s > d - \delta}) = B_{\bP_\eps(s) > \bP_\eps(d)-\bP_\eps(\delta)}\neq \{0\}. 
\]
Now we will obtain a contradiction by showing that
\begin{itemize}
\item[$(i)$] $B_1 \defeq B_{\eps < \frac12 ( M(b) - f(d))} \cap B_{s > d - \delta} \cap B_{d < b} = \{0\}$, \text{ and} 
\item[$(ii)$] $B_2 \defeq B_{\eps < \frac12 (M(b) - f(d))} \cap B_{s > d - \delta} \cap B_{d = b} = \{0\}$.
\end{itemize}
These statements combined yield a contradiction, since 
\[
B_1 + B_2 = B_{\eps < M(b) - f(d)} \cap B_{s > d - \delta} \neq \{0\}.
\]
We first show (i). To this end, let $x \in [a, (\bP_1 (d + \delta) + \bP_1^d(s)) \wedge b]$, then
\[ \bP_1 \bP_{x \leq d - \delta} (a) \leq \bP_1 \bP_{x \leq d - \delta} (x) \leq \bP_1 \bP_{x \leq d - \delta} (d - \delta) \leq \bP_1 \bP_{x \leq d - \delta} (s). \]
Thus, by \Cref{l:restricting_to_bands}(i), we have that
\[ \bP_1 \bP_{x \leq d - \delta} (f(x)) \leq \bP_1 \bP_{x \leq d - \delta} (M(s)) = \bP_1 \bP_{x \leq d - \delta} (M(b) - (M(b) - M(s))). \]
Since $\bP_{x > d - \delta} (d - \delta) \leq \bP_{x > d - \delta}(x)$ and $\bP_1 (x) \leq \bP_1( d + \delta)$, we have that
\[ \bP_1 \bP_{x > d - \delta} (d - \delta) \leq \bP_1 \bP_{x > d - \delta} (x) \leq \bP_1 \bP_{x > d - \delta} (d + \delta), \]
and so $\bP_1 \bP_{x > d - \delta} \left( \abs{x - d} \right) \leq \delta$. By \Cref{l:restricting_to_bands}(iii), 
\[
\bP_1 \bP_{x > d - \delta} \left( \abs{f(x) - f(d)} \right) \leq \bP_1 \bP_{x > d - \delta}(\eps) 
\]
and so
\begin{align*}
\bP_1 \bP_{x > d - \delta} (f(x))
& \leq \bP_1 \bP_{x > d - \delta} (f(d) + \eps) \\
& \leq \bP_1 \bP_{x > d - \delta} (f(d) + \tfrac12 (M(b) - f(d)) \\
& = \bP_1 \bP_{x > d - \delta} \left(M(b) - \tfrac12 (M(b) - f(d))\right). 
\end{align*}
Since $\bP_1^d(x) \in \bP_1^d ([a, s])$ and $f$ is locally band preserving,
\[ \bP_1^d (f(x)) \leq \bP_1^d (M(s)) = \bP_1^d (M(b) - (M(b) - M(s))) \]
holds by \Cref{l:restricting_to_bands}(i). Therefore,
\begin{align*}
f(x)
& \leq \left( \bP_1 \bP_{x \leq d - \delta} + \bP_1^d \right) (M(b) - (M(b) - M(s))) \\ 
& \hskip .8cm + \bP_1 \bP_{x > d - \delta} \left(M(b) - \tfrac12 (M(b) - f(d))\right) \\
& \leq M(b) - (M(b) - M(s)) \wedge \tfrac12 (M(b) - f(d)).
\end{align*}
By taking the supremum over all $x \in [a, (\bP_1 (d + \delta) + \bP_1^d(s)) \wedge b]$, we have that
\[ M(b) - M((\bP_1 (d + \delta) + \bP_1^d(s)) \wedge b) \geq (M(b) - M(s)) \wedge \tfrac12 (M(b) - f(d)). \]
Therefore, $M(b) - M((\bP_1 (d + \delta) + \bP_1^d(s)) \wedge b)$ is a weak order unit, so $(\bP_1 (d + \delta) + \bP_1^d(s)) \wedge b \in S$. Then we have $(\bP_1 (d + \delta) + \bP_1^d(s)) \wedge b \leq d$, hence 
\[
\bP_1(d) \le (\bP_1(d) + \bP_1(\delta)) \wedge \bP_1(b) \le \bP_1(d).
\]
Thus $\bP_1(d) = (\bP_1(d) + \bP_1(\delta)) \wedge \bP_1(b)$, so

\[
\bP_1(d) = \bP_{b=d}\bP_1(b) + \bP_{d<b}\bP_1(d+\delta) = \bP_1(d+\delta),
\]
as $B_1 \subseteq B_{d < b} = B_{b=d}^d$. Hence $\bP_1(\delta) = 0$, showing that $B_1 = \{0\}$.

Now we show (ii). To this end, let $x \in [a, b]$. We have that
\[ \bP_2 \bP_{x \leq d - \delta} (a) \leq \bP_2 \bP_{x \leq d - \delta} (x) \leq \bP_2 \bP_{x \leq d - \delta} (d - \delta) \leq \bP_2 \bP_{x \leq d - \delta} (s). \]
By \Cref{l:restricting_to_bands}(i),
\[  \bP_2 \bP_{x \leq d - \delta} (f(x)) \leq  \bP_2 \bP_{x \leq d - \delta} (M(s)) =  \bP_2 \bP_{x \leq d - \delta} ( M(b) - (M(b) - M(s)) ). \]
Since
\[ \bP_2 \bP_{x > d - \delta} (d - \delta) \leq \bP_2 \bP_{x > d - \delta} (x) \leq \bP_2 \bP_{x > d - \delta} (b) = \bP_2 \bP_{x > d - \delta} (d), \]
we have that $\bP_2 \bP_{x > d - \delta} \left( \abs{x - d} \right) \leq \delta$, so we can apply \Cref{l:restricting_to_bands}(iii) to get that
\begin{align*}
\bP_2 \bP_{x > d - \delta} (f(x))
& \leq \bP_2 \bP_{x > d - \delta} (f(d) + \eps) \\
& \leq \bP_2 \bP_{x > d - \delta} (f(d) + \tfrac12 (M(b) - f(d))) \\
& = \bP_2 \bP_{x > d - \delta} \left(M(b) - \tfrac12 (M(b) - f(d))\right).
\end{align*}
Therefore,
\begin{align*}
&\bP_2 (f(x))= \bP_2 \bP_{x \leq d - \delta} (f(x)) + \bP_2 \bP_{x > d - \delta} (f(x)) \\
& \leq \bP_2 \bP_{x \leq d - \delta} ( M(b) - (M(b) - M(s)) ) + \bP_2 \bP_{x > d - \delta} \left(M(b) - \tfrac12 (M(b) - f(d))\right)  \\
& \leq \bP_2 (M(b) - (M(b) - M(s)) \wedge \tfrac12 (M(b) - f(d))).
\end{align*}
By taking the supremum over all $x \in [a, b]$, we have that 
\[
\bP_2 (M(b)) \leq \bP_2 (M(b) - (M(b) - M(s)) \wedge \tfrac12 (M(b) - f(d))).
\]
This inequality shows that $\bP_2 ((M(b) - M(s)) \wedge \frac12 (M(b) - f(d))) = 0$. Now 
\[
(M(b) - M(s)) \wedge \tfrac12 (M(b) - f(d))
\]
is a weak order unit, so $B_2 = \{0\}$.

Therefore, the assumptions $f(a) \ll M(b)$ and $f(d) \ll M(b)$ lead to a contradiction. If $B \defeq B_{f(a) < M(b)} \cap B_{f(d) < M(b)} \neq \{0\}$, we can restrict ourselves to $B$ by considering the function $g\colon \bP([a, b]) \to B$ given by 
\[
g(x) := \bP (f(x + \bP^d(a))). 
\]
Since $g(\bP(a)) \ll \bP (M(b))$ and, using that $f$ is locally band preserving, $g(\bP(d)) = \bP(f(d)) \ll \bP(M(b))$ in $B$, the above argument yields a contradiction. Hence, it must be that $B  = \{0\}$, and so we obtain the inclusion $B_{f(a) < M(b)} \subseteq B_{f(d) = M(b)}$, which means $\bP_{f(a) < M(b)} (f(d)) = \bP_{f(a) < M(b)} (M(b))$. Let $d_0 := \bP_{f(a) = M(b)} (a) + \bP_{f(a) < M(b)} (d)$. Since $f$ is locally band preserving,
\begin{align*}
f(d_0)
& = \bP_{f(a) = M(b)} (f(d_0)) + \bP_{f(a) < M(b)} (f(d_0)) \\
& = \bP_{f(a) = M(b)} (f(a)) + \bP_{f(a) < M(b)} (f(d)) \\
& = \bP_{f(a) = M(b)} (M(b)) + \bP_{f(a) < M(b)} (M(b)) \\
& = M(b).
\end{align*}
Hence for all $x \in [a,b]$ we have shown that $f(x) \le f(d_0)$. The existence of $c$ follows by considering $-f$.
\end{proof}
We now prove an analogue of Rolle's Theorem and the Mean Value Theorem for super order differentiable functions. Note that by \Cref{t:differentiable_implies_lbp} all such functions are automatically locally band preserving. 

\begin{theorem}[Rolle's Theorem] \label{t:RT}
Let $E$ be a Dedekind complete real $\Phi$-algebra. Consider $a,b\in E$ with $a\ll b$. Let $f\colon [a, b] \to E$ be order continuous on $[a, b]$ and super order differentiable on $(a, b)$. If $f(a) = f(b)$, then there exists $x_0 \in (a, b)$ with $f'(x_0) = 0$.
\end{theorem}
\begin{proof}
By \Cref{t:EVT}, there exist $c, d \in [a, b]$ such that $f(c) \leq f(x) \leq f(d)$ holds for all $x \in [a, b]$. Then $E = B_0 \oplus B_1 \oplus B_2$, where
\begin{align*}
B_0 & \defeq B_{a < d} \cap B_{d < b}, \\
B_1 & \defeq B_0^d \cap B_{a < c} \cap B_{c < b}, \text{ and} \\
B_2 & \defeq B_0^d \cap B_1^d.
\end{align*}
It follows that
\begin{align*}
B_2
& = (B_{a < d} \cap B_{d < b})^d \cap (B_0^d \cap B_{a < c} \cap B_{c < b})^d \\
& = (B_{a = d} + B_{d = b}) \cap (B_0 + B_{a = c} + B_{c = b}) \\
& = (B_{a = d} \cap B_{a = c}) + (B_{a = d} \cap B_{c = b}) + (B_{d = b} \cap B_{a = c}) + (B_{d = b} \cap B_{c = b}) \\
& = (B_{a = d} + B_{d = b}) \cap (B_{a = c} + B_{c = b}).
\end{align*}

Let $x_0 \defeq \bP_0(d) + \bP_1(c) + \bP_2(\frac12(a+b))$, and note that $x_0 \in (a, b)$. Since 
$\bP_{a = d} (f(d)) = \bP_{a = d} (f(a))$ and 
\[
\bP_{d = b} (f(d)) = \bP_{d = b} (f(b)) = \bP_{d = b} (f(a)), 
\]
we have that $\bP_2 (f(d)) = \bP_2 (f(a))$. Similarly, $\bP_2 (f(c)) = \bP_2 (f(a))$. Thus $\bP_2 \circ f$ is constant, so $\bP_2 (f'(x_0)) = 0$. We will show that $\bP_0 (f'(x_0)) = 0$. The argument for why $\bP_1 (f'(x_0)) = 0$ is similar as in both cases $x_0$ equals an extremum on the band.

Since $f$ is super order differentiable at $x_0$, there exists an $r \gg 0$ such that $(x_0 - r, x_0 + r) \subseteq [a, b]$ and there exists an $\Eps \searrow 0$ such that for every $\eps \in \Eps$ there exists an $N_\eps \in \bN$ satisfying
\[ y \in [a, b] \text{ and } \abs{y - x_0} \leq \tfrac{1}{N_\eps} r \implies \abs{f(y) - f(x_0) - (y - x_0)f'(x_0)} \leq \abs{y - x_0} \eps. \]

For $n \in \bN$, define $y_n := x_0 + \frac{1}{n+1} r$. Then $y_n \in [a, b]$ for all $n \in \bN$.
Let $m \in \bN$. Since $\frac{1}{m}e$ is invertible, we have that $\bP(r)$ has an inverse in $B \defeq B_{\frac{1}{m}e \leq r}$ by \cite[Theorem~11.1]{depagter}. Denote this inverse by $s$. Then, for any $\eps \in \Eps$ and $n \geq N_\eps$, $\abs{y_n - x_0} \leq \frac{1}{N_\eps} r$ so
\[ \abs*{f(y_n) - f(x_0) - \tfrac{1}{n} r f'(x_0) } \leq \abs*{y_n - x_0} \eps \leq \tfrac{1}{n+1} r \eps, \]
and hence, by applying the projection $\bP_0 \bP$ and multiplying by $(n+1)s$ on both sides,
\[ \abs*{\bP_0 \bP (f(y_n) - f(x_0)) (n + 1)s - \bP_0 \bP(f'(x_0))} \leq \bP_0 \bP(\eps). \]
Thus, $\bP_0 \bP (f(y_n) - f(x_0)) (n + 1)s \to \bP_0 \bP (f'(x_0))$. For all $n \in \bN$, since we have $\bP_0 (x_0) = \bP_0 (d)$, and $f$ is locally band preserving,
\[ \bP_0 \bP (f(y_n) - f(x_0)) (n + 1)s = \bP_0 \bP (f(y_n) - f(d)) (n + 1)s \leq 0. \]
Thus, we have $\bP_0 \bP (f'(x_0)) \leq 0$. Replacing $y_n$ by $z_n := x_0 - \frac{1}{n+1} r$ in the argument above shows that $\bP_0 \bP (f'(x_0)) \geq 0$. Therefore, $\bP_0 \bP (f'(x_0)) = 0$. Since $m$ was arbitrary, by \Cref{l:projection_weak_unit_invertible}, we have that $\bP_0 (f'(x_0)) = 0$. 
\end{proof}

\begin{theorem}[Mean Value Theorem] \label{t:MVT}
Let $E$ be a Dedekind complete real $\Phi$-algebra. Consider $a,b\in E$ with $a\ll b$. Let $f\colon [a, b] \to E$ be order continuous on $[a, b]$ and super order differentiable on $(a, b)$. Then there exists an $x_0 \in (a, b)$ such that
\[ (b - a) f'(x_0) = f(b) - f(a). \]
\end{theorem}
\begin{proof}
Define $g(x) \defeq (b - a) f(x) - (f(b) - f(a)) x$. Then $g$ is order continuous on $[a, b]$ and super order differentiable on $(a, b)$ since $f$ and $x \mapsto x$ are. Note that $x \mapsto x$ has derivative $e$ everywhere. Also,
\[ g(a) = (b - a) f(a) - (f(b) - f(a)) a  = (b - a) f(b) - (f(b) - f(a)) b = g(b).  \]
By \Cref{t:RT}, there exists an $x_0 \in (a, b)$ such that
\[ 0 = g'(x_0) = (b - a)f'(x_0) - (f(b) - f(a))e.\]
Therefore, $(b - a) f'(x_0) = f(b) - f(a)$.
\end{proof}

\begin{corollary} \label{c:MVT_cor}
Let $E$ be a Dedekind complete real $\Phi$-algebra. Consider $a,b\in E$ with $a\ll b$. Let $f\colon (a,b)\to E$ be super order differentiable. For any $c,d\in(a,b)$, there exists $x_0\in\{c+x(d-c)\colon x\in(0,e)\}$ such that
\[
(d-c)f'(x_0)=f(d)-f(c).
\]
\end{corollary}

\begin{proof}
Define $g \colon [0, e] \to E$ by $g(x) := f(c_x)$, where $c_x \defeq c + x(d-c)$. It is easily verified that $g$ is order continuous on $[0, e]$, and we will show that $g$ is super order differentiable on $(0, e)$ with $g'(x) = (d-c) f'(c_x)$. To that end, let $\Delta \downarrow 0$ and $x \in (0, e)$. Since $f$ is super order differentiable at $c_x$, there exists an $\Eps \searrow 0$ such that for every $\eps \in \Eps$ there exists a $\delta \in \Delta$ such that for $y \in (a,b)$ the implication
\begin{align*}
\abs{y - c_x} \leq \abs{d - c} \delta \implies \abs*{f(y) - f(c_x) - (y - c_x) f'(c_x)} \leq \abs{y - c_x} \eps 
\end{align*}
holds. Let $y \in [0, e]$ be such that $\abs{y - x} \leq \delta$. Then $c_y \in (a,b)$ and 
\[
\abs{c_y - c_x} = \abs{d - c} \abs{y - x} \leq \abs{d-c} \delta,
\]
so that
\begin{align*}
&\abs*{g(y) - g(x) - (y - x) (d-c)f'(c_x)} \\
& \hskip 2.15cm = \abs*{f(c_y) - f(c_x) - (y - x) (d - c) f'(c_x)} \\
& \hskip 2.15cm = \abs*{f(c_y) - f(c_x) - (c_y - c_x)f'(c_x)} \\
& \hskip 2.15cm \leq \abs{c_y - c_x} \eps \\
& \hskip 2.15cm = \abs{x - y} \abs{d - c} \eps.
\end{align*}
Therefore, $g$ is super order differentiable on $(0, e)$. Now we apply \Cref{t:MVT} to obtain a $s \in (0, e)$ such that
\[(d-c) f'(c_s) = g'(s) = g(e) - g(0) = f(d) - f(c). \]
Hence, we may choose $x_0 \defeq c_s$.
\end{proof}

\begin{corollary} \label{c:zero_derivative_constant}
Let $E$ be a Dedekind complete real $\Phi$-algebra. Consider $a,b\in E$ with $a\ll b$. Let $f\colon (a, b) \to E$ be super order differentiable with $f' = 0$. Then $f$ is constant.
\end{corollary}
\begin{proof}
For any $c, d \in (a, b)$, by applying \Cref{c:MVT_cor}, there exists an $x_0 \in (a, b)$ such that $f(d) - f(c) = (d - c) f'(x_0) = 0$ and hence $f(c) = f(d)$.
\end{proof}

\Cref{c:zero_derivative_constant} requires $f$ to be super order differentiable, and the following example shows that it fails if $f$ is only order differentiable. It also shows that a function can be order differentiable but not locally band preserving.

\begin{example}
Consider $E = \bR^\bN$, and define $f \colon (0, e) \to \bR^\bN$ by
\[ f((x_n)_{n \in \bN}) := 
\begin{cases}
0 & \text{ if } x_n \to \frac12, \\
e & \text{ if } x_n \not\to \frac12.
\end{cases} \]

Let $x \in (0, e)$. Define $r := (\frac{1}{n} \wedge (1 - x_n) \wedge x_n)_{n \in \bN}$. It follows that $(x - r, x + r) \subseteq (0, e)$, and for any $y \in (x - r, x+r)$, $\abs{x_n - y_n} < r_n \leq \frac{1}{n}$ holds for all $n \in \bN$, so that $y_n \to \frac12$ if and only if $x_n \to \frac12$. Therefore, $f$ is constant on $(x - r, x + r)$. By \Cref{p:differentiation_local} and the fact that constant functions have derivative $0$, $f$ is order differentiable on $(0, e)$ with $f' = 0$.

To see that $f$ is not locally band preserving, consider $e_1 := (1, 0, 0, \dots)$ and note that $\bP_{e_1}(\frac12 e_1) = \frac12 e_1 = \bP_{e_1}(\frac12e)$, but 
\[
\bP_{e_1}(f( \tfrac12 e_1)) = \bP_{e_1}(e) = e_1 \neq 0 = \bP_{e_1}(f(\tfrac12e)).
\]
\end{example}

\begin{remark}
    Note that for a Dedekind complete real $\Phi$-algebra $E$, with $a, b \in E$ such that $a \ll b$, and $f \colon [a,b] \to E$ that is order continuous and super order differentiable on $(a,b)$, \Cref{c:MVT_cor} implies the following characteristic properties for $f$. In statement $(ii)$ it is used that the product of two weak order units in a $\Phi$-algebra is a weak order unit. This fact follows from \cite[Theorem~142.5]{zaanen2} and \cite[Theorem~142.3]{zaanen2}.
    \begin{itemize}
        \item[$(i)$] If $f'(x) \ge 0$ for all $x \in (a,b)$, then $f$ is increasing on $[a,b]$, 
        \item[$(ii)$] if $f'(x) \gg 0$ for all $x \in (a,b)$, then  for any $c,d \in [a,b]$ with $c \ll d$ we have that $f(c) \ll f(d)$,
        \item[$(iii)$] if $f’(x)>0$ for all $x \in (a,b)$, then for any $c,d \in [a,b]$ with  $c \ll d$, we have that $f(c)<f(d)$, and
        \item[$(iv)$] if $f’(x) \gg 0$ for all $x \in (a,b)$, then for any $c,d \in [a,b]$ with $c<d$, we have that $f(c)<f(d)$.
    \end{itemize}
\end{remark}

Finally, we return to the complex case. To prove that a zero super order derivative implies a function is constant on order open neighbourhoods, we first prove a complex version of the Mean Value Theorem.

\begin{theorem}[Complex Mean Value Theorem] \label{p:complex_MVT}
Let $E$ be a Dedekind complete complex $\Phi$-algebra. Consider $c \in E$ and a weak order unit $r$. Let $f\colon N(c, r) \to E$ be super order differentiable. Then for any $a, b \in N(c, r)$, there exist $u, v \in \{a + x(b-a): x \in (0, e)\}$ such that
\begin{itemize}
\item[$(i)$]$\Re((b - a)f'(u))  = \Re(f(b) - f(a))$, and
\item[$(ii)$]$\Im((b-a)f'(v))  = \Im(f(b) - f(a))$.
\end{itemize}
\end{theorem}
\begin{proof}
The proof is similar to that of \Cref{c:MVT_cor}.
% Define $g \colon [0, e] \to E_\rho$ by $g(x) := \Re (f(a_x))$, where $a_x \defeq a + x(b-a)$. It is easily verified that $g$ is order continuous on $[0, e]$, and we will show that $g$ is super order differentiable on $(0, e)$ with $g'(x) = \Re ((b-a) f'(a_x))$. To that end, let $\Delta \downarrow 0$ and $x \in (0, e)$. Since $f$ is super order differentiable at $a_x$, there exists an $\Eps \searrow 0$ such that for every $\eps \in \Eps$ there exists a $\delta \in \Delta$ such that for $z \in N(c, r)$ the implication
% \begin{align*}
% \abs{z - a_x} \leq \abs{b - a} \delta \implies \abs*{f(z) - f(a_x) - (z - a_x) f'(a_x)} \leq \abs{z - a_x} \eps 
% \end{align*}
% holds. Let $y \in [0, e]$ be such that $\abs{y - x} \leq \delta$. Then $a_y \in N(c, r)$ and 
% \[
% \abs{a_y - a_x} = \abs{b - a} \abs{x - y} \leq \abs{b-a} \delta,
% \]
% so that
% \begin{align*}
% &\abs*{g(y) - g(x) - (y - x) \Re ( (b-a)f'(a_x))} \\
% & \hskip 2.15cm = \abs*{\Re( f(a_y) - f(a_x) - (y - x) (b - a) f'(a_x))} \\
% & \hskip 2.15cm \leq \abs*{f(a_y) - f(a_x) - (y - x) (b - a) f'(a_x)} \\
% & \hskip 2.15cm = \abs*{f(a_y) - f(a_x) - (a_y - a_x)f'(a_x)} \\
% & \hskip 2.15cm \leq \abs{a_y - a_x} \eps \\
% & \hskip 2.15cm = \abs{x - y} \abs{a - b} \eps.
% \end{align*}
% Therefore, $g$ is super order differentiable on $(0, e)$. Now we apply \Cref{t:MVT} to obtain a $c \in (0, e)$ such that
% \[ \Re((b-a) f'(a_c)) = g'(c) = g(e) - g(0) = \Re(f(b) - f(a)). \]
% Hence, we may choose $u \defeq a_c$. A similar argument shows the existence of $v$.
\end{proof}

\begin{corollary}
Let $E$ be a Dedekind complete complex $\Phi$-algebra. Consider $c \in E$ and $r$ a weak order unit. Let $f: N(c, r) \to E$ be super order differentiable with $f' = 0$. Then $f$ is constant.
\end{corollary}
\begin{proof}
Suppose that $a, b \in N(c, r)$. It follows from \Cref{p:complex_MVT} that there exist $u, v \in \{a + x(b-a): x \in (0, e)\}$ such that \[\Re(f(b) - f(a)) = \Re((b-a) f'(u)) = 0,
\]
\text{ and }
\[\Im (f(b) - f(a)) = \Im((b-a) f'(v) = 0.
\]
Thus, $f(a) = f(b)$.
\end{proof}

\noindent
{\bf Acknowledgements.} The authors would like to express their gratitude to the Erasmus$+$ ICM Grant for facilitating the productive visits between Leiden University and the University of Pretoria. We also thank Onno van Gaans for the fruitful discussions during these visits, Eugene Bilokopytov for his help in constructing \Cref{e:cont_bounded_fail}, and Denny Leung for providing \Cref{E:pointwise}.

\bibliographystyle{alpha}
\bibliography{refs.bib}

\begin{thebibliography}{Muj86}

\bibitem[AB06]{aliprantis}
C.D. Aliprantis and O.~Burkinshaw.
\newblock {\em Positive {O}perators}, volume 119.
\newblock Springer Science \& Business Media, 2006.

\bibitem[AZ18]{Aviles-Zapata}
A.~Avil\'{e}s and J.M. Zapata.
\newblock Boolean-valued models as a foundation for locally {$L^0$}-convex analysis and conditional set theory.
\newblock {\em J. Appl. Logics}, 5(1):389--419, 2018.

\bibitem[dP81]{depagter}
B.~de~{P}agter.
\newblock {\em $f$-algebras and Orthomorphisms}.
\newblock PhD thesis, Rijksuniversiteit te Leiden, 1981.

\bibitem[dS73]{Schipper}
W.J. de~Schipper.
\newblock A note on the modulus of an order bounded linear operator between complex vector lattices.
\newblock {\em Indag. Math.}, 35:355--367, 1973.
\newblock Nederl. Akad. Wetensch. Proc. Ser. A {\bf 76}.

\bibitem[EW98]{Wickstead-Zafer}
Z.~Ercan and A.W. Wickstead.
\newblock Towards a theory of nonlinear orthomorphisms.
\newblock In {\em Functional analysis and economic theory ({S}amos, 1996)}, pages 65--78. Springer, Berlin, 1998.

\bibitem[LZ71]{zaanen1}
W.A.J. Luxemburg and A.C. Zaanen.
\newblock {\em Riesz Spaces I}.
\newblock North-Holland Publishing Co., 1971.

\bibitem[Muj86]{Mujica}
J.~Mujica.
\newblock {\em Complex {A}nalysis in {B}anach {S}paces}, volume 120 of {\em North-Holland Mathematics Studies}.
\newblock North-Holland Publishing Co., Amsterdam, 1986.
\newblock Holomorphic functions and domains of holomorphy in finite and infinite dimensions, Notas de Matem\'atica, 107. [Mathematical Notes].

\bibitem[RS19]{series}
M.~Roelands and C.~Schwanke.
\newblock Series and power series on universally complete complex vector lattices.
\newblock {\em J. Math. Anal. Appl.}, 473(2):680--694, 2019.

\bibitem[RS25]{differentiation}
M.~Roelands and C.~Schwanke.
\newblock Differentiable, holomorphic, and analytic functions on complex phi-algebras.
\newblock {\em J. Math. Anal. Appl.}, 541(1):128671, 2025.

\bibitem[Zaa83]{zaanen2}
A.C. Zaanen.
\newblock {\em Riesz Spaces II}.
\newblock North-Holland Publishing Co., 1983.

\bibitem[Zaa97]{babyzaanen}
A.C. Zaanen.
\newblock {\em Introduction to Operator Theory in Riesz Spaces}.
\newblock Springer, 1997.

\end{thebibliography}

\end{document}